\documentclass[12pt,leqno]{article}
\usepackage{amsfonts}
\usepackage{amsmath, amsthm, amsfonts, amssymb, color}
\usepackage{mathrsfs}
\renewcommand{\bar}{\overline}
\renewcommand{\hat}{\widehat}
\renewcommand{\tilde}{\widetilde}

\newtheorem{thm}{Theorem}[section]

\newtheorem{lem}[thm]{Lemma}
\newtheorem{prp}[thm]{Proposition}

\newtheorem{exa}[thm]{Example}
\theoremstyle{definition}
\newtheorem{defn}{Definition}[section]
\newcommand{\scr}[1]{\mathscr #1}
\definecolor{wco}{rgb}{0.5,0.2,0.3}

\numberwithin{equation}{section} \theoremstyle{remark}

\newcommand{\ua}{\uparrow}

\title{{\bf Ergodicity for Neutral Type  SDEs   with Infinite Length of  Memory}
\thanks{This work is
supported   in part by NNSFC (11771326,  11431014, 11726627).} }
\author{
{\bf  Jianhai Bao$^{b),c)}$,  Feng-Yu Wang$^{a),c)}$, Chenggui Yuan$^{c)}$}\\
\footnotesize{$^{a)}$Center for Applied Mathematics, Tianjin
University, Tianjin 300072, China}\\
\footnotesize{$^{b)}$School of Mathematics and Statistics, Central
South
University, Changsha 410083, China}\\
\footnotesize{$^{c)}$Department of Mathematics, Swansea University,
Singleton Park, SA2 8PP, UK}\\ \footnotesize{jianhaibao@csu.edu.cn,
wangfy@bnu.edu.cn, C.Yuan@swansea.ac.uk}}

\begin{document}
\def\R{\mathbb R}  \def\ff{\frac} \def\ss{\sqrt} \def\B{\mathbf
B}
\def\N{\mathbb N} \def\kk{\kappa} \def\m{{\bf m}}
\def\dd{\delta} \def\DD{\Delta} \def\vv{\varepsilon} \def\rr{\rho}
\def\<{\langle} \def\>{\rangle} \def\GG{\Gamma} \def\gg{\gamma}
  \def\nn{\nabla} \def\pp{\partial} \def\EE{\scr E}
\def\d{\text{\rm{d}}} \def\bb{\beta} \def\aa{\alpha} \def\D{\scr D}
  \def\si{\sigma} \def\ess{\text{\rm{ess}}}
\def\beg{\begin} \def\beq{\begin{equation}}  \def\F{\scr F}
\def\Ric{\text{\rm{Ric}}} \def\Hess{\text{\rm{Hess}}}
\def\e{\text{\rm{e}}} \def\ua{\underline a} \def\OO{\Omega}  \def\oo{\omega}
 \def\tt{\tilde} \def\Ric{\text{\rm{Ric}}}
\def\cut{\text{\rm{cut}}} \def\P{\mathbb P} \def\ifn{I_n(f^{\bigotimes n})}
\def\C{\scr C}      \def\aaa{\mathbf{r}}     \def\r{r}
\def\gap{\text{\rm{gap}}} \def\prr{\pi_{{\bf m},\varrho}}  \def\r{\mathbf r}
\def\Z{\mathbb Z} \def\vrr{\varrho} \def\ll{\lambda}
\def\L{\scr L}\def\Tt{\tt} \def\TT{\tt}\def\II{\mathbb I}
\def\i{{\rm in}}\def\Sect{{\rm Sect}}\def\E{\mathbb E} \def\H{\mathbb H}
\def\M{\scr M}\def\Q{\mathbb Q} \def\texto{\text{o}} \def\LL{\Lambda}
\def\Rank{{\rm Rank}} \def\B{\scr B} \def\i{{\rm i}} \def\HR{\hat{\R}^d}
\def\to{\rightarrow}\def\l{\ell}
\def\8{\infty}\def\X{\mathbb{X}}\def\3{\triangle}
\def\V{\mathbb{V}}\def\M{\mathbb{M}}\def\W{\mathbb{W}}\def\Y{\mathbb{Y}}\def\1{\lesssim}

\def\La{\Lambda}\def\S{\mathbf{S}}

\renewcommand{\bar}{\overline}
\renewcommand{\hat}{\widehat}
\renewcommand{\tilde}{\widetilde}
 \maketitle

\begin{abstract}
In this paper,  the weak
Harris theorem developed in \cite{HMS11} is illustrated  by using a
straightforward  Wasserstein coupling, which implies the exponential
ergodicity of the functional solutions to a range of neutral type
SDEs with infinite length of memory. A concrete example is presented
to illustrate the main result.

\noindent
 AMS Subject Classification:\  34K50, 37A30, 60J05  \\
\noindent
 Keywords: ergodicity, neutral type stochastic differential equation, infinite memory,
 weak Harris' theorem, Wasserstein coupling
 
 \end{abstract}
 \vskip 2cm

\section{Introduction}

The erodicity theory is a rich and active area in the study of
Markov processes and related topics.  Existing results include both
qualitative characterizations (for instance, existence and
uniqueness of invariant probability measures, strong Feller
property, irreducibility) and quantitative estimates (convergence
rate of Markov transition semigroups, gradient and heat kernel
estimates, etc.). Among many other references, we would like  to
mention
  \cite{CL,DZ,H16,H08,MT93,Wang04} for the study of non-degenerate stochastic differential equations
(SDEs) and stochastic partial differential equations (SPDEs), and  \cite{BGM,DMS13,GS,GW,HM,MSH,Vil09,W17} for
 degenerate SDEs/SPDEs. In these references, several different probability distances (for example, total variational distance, $L^2$ distance,
 and  Warsserstein distance) have been adopted to measure the convergence rate of
 Markov transition semigroups. Efficient tools developed in the literature include  functional inequalities
 (for instance, weak Poincar\'e, Poincar\'e, and log-Sobolev inequalities), Lyapunov type criteria, Harris' theorem, and coupling method, etc.

For path-dependent SDEs (i.e., the coefficients depend on the
history), which are also called functional SDEs or SDEs with memory,
the solutions are no longer Markovian. In this case,  one
investigates the functional solutions (i.e., the segment process,
also called window process, of the solutions), which are Markov
processes on the path space determined by the length of memory.
However, the above
 tools mentioned are very hard to apply to this type
infinite-dimensional Markov processes: \beg{enumerate}
\item[$\bullet$] Due to the lack of characterization on Dirichlet
forms, functional inequalities are not yet established;
\item[$\bullet$] The Lyapunov condition on the path space  is less explicit  since the formulation of infinitesimal generator  is not yet available;
on account of  the same reason,  the classical coupling argument via
coupling operator is invalid;
\item[$\bullet$] Since the functional solutions are highly degenerate (infinite-dimensional Markov processes with finite-dimensional noises),
the classical Harris' theorem does not apply.
  \end{enumerate}

In recent years, some new approaches have been developed to
investigate the ergodicity and related properties for path-dependent
SDEs. When the noise term is path-independent and the drift depends
only on a finite segment  of path, the ergodicity under the total
variational distance was investigated in \cite{BS}, while gradient
estimates and Harnack type inequalities (which in particular imply
the strong Feller and irreducibility) have been established in
\cite{BWY11,BWY11b,WY10,Wbook}, to name a few, by using coupling by
change of measures. When the noise part is also path-dependent, but
both drift and noise parts depend only on  a fixed length of past
path, a weak Harris' theorem has been established in \cite{HMS11} to
derive the exponential ergodicity under
  the Wasserstein distance. In case the noise is path-dependent,  we would like to emphasize  that  the ergodicity under the total variational distance and the strong
    Feller property are not available  since   the laws of functional solutions with different initial data are mutually singular.
    The weak Harris' theorem has  been applied in, e.g., \cite{B14,CH,MT,TM} to establish the ergodicity for highly degenerate stochastic dynamical systems including Markov processes with random switching.

 In this paper, we  aim to  investigate  the
exponential ergodicity  for path-dependent SDEs of neutral type and
with infinite length of memory. Such kind of model  fits more real
world systems whose time evolution depends on the whole history (cf.
\cite[Chapter 6]{M08}). Intuitively, the farer the history, the
weaker the influence to the evolution of the system. So, in the
following we will take a reference norm on the path space which
indicates that the influence of history exponentially  decay when
the time goes to $-\infty$.

 For an integer $d\ge1$, let $(\R^d,\<\cdot,\cdot\>,|\cdot|)$ be  the standard
$d$-dimensional Euclidean space, and
$\R^d\otimes\R^d$  the family of all $d\times d$-matrices
 equipped with the Hilbert-Schmidt norm
$\|\cdot\|_{\rm HS}$.  $\C=C((-\8,0];\R^d)$ stands for   the space
of all continuous maps $f:(-\8,0]\rightarrow\R^d$.  For a map
$f(\cdot): (-\infty,\infty)\to \R^d$, define its segment map
$f_\cdot: [0,\infty) \to   \C$  by setting
$$f_t(\theta)= f(t+\theta),\ \ t\ge 0, \theta\in (-\infty,0].$$
For  a
fixed number
 $r\in(0,\8),$ let
\begin{equation*}
\mathscr{C}_r=\Big\{\phi\in
\C:\|\phi\|_r:=\sup_{-\8<\theta\le0}(\e^{r\theta}|\phi(\theta)|)<\8\Big\}.
\end{equation*}
Then,  $(\C_r, \|\cdot\|_r)$  is
 a Polish  space.  The norm $\|\cdot\|_r$ fits the intuition of exponential decay with regard to the influence of history; that is,  the contribution to
  the norm from   the   history at time $\theta<0$ has a minus exponential discount $\e^{\theta
  r}$. For any $\theta\in(-\8,0]$, let $\xi_0(\theta)\equiv{\bf0},$  a $d$-dimensional zero
  vector.

Consider the following path-dependent SDE on $\R^d$ of neutral type
\begin{equation}\label{eq1}
\d \{X(t)-G(X_t)\}=b(X_t)\d t+\si(X_t)\d W(t),~~~t>0, ~~~X_0=\xi\in\C_r,
\end{equation}
where $G,b:\C_r\to\R^d$ and $\si:\C_r\to\R^d\otimes\R^d$ are
measurable with $G(\xi_0)={\bf0}$, $(X_t)_{t\ge0}$ is the segment
process associated with   $(X(t))_{t\ge0}$, $(W(t))_{t\ge0}$ is the
$d$-dimensional Brownian motion  on a complete filtration
probability space $(\OO,\F, (\F_t)_{t\ge 0},\P)$.
 For more motivating
examples of \eqref{eq1}, please refer to \cite[p.201-2]{M08}.

A continuous adapted process $(X(t))_{t\ge 0}$ is called a solution
to \eqref{eq1} with the initial value $X_0$, if $\P$-a.s.
$$X(t)= X(0)+ G(X_t)-G(X_0)+\int_0^t b(X_s)\d s + \int_0^t \si(X_s) \d W(s),\ \ t\ge 0.$$ 
We call $(X_t^\xi)_{t\ge 0}$ a functional solution to \eqref{eq1}
with the initial value $X_0^\xi=\xi\in\C_r$.

To investigate  the exponential convergence of $P_t(\xi,\cdot)$
under the metric $\mathbb W_{\rr_r}$, we   impose the following
assumptions.
 \begin{enumerate}
 \item[({\bf A0})]  $b$ and $\si$ are continuous and bounded on bounded subsets of $\C_r$, and there exists $\aa\in (0,1)$ such that
$$|G(\xi)-G(\eta)|\le \aa\|\xi-\eta\|_r,\ \ \xi,\eta\in \C_r.$$
\item[({\bf A1})] There exists a constant $L_0>0$ such that
 $$ \<\xi(0)-\eta(0)+G(\eta)-G(\xi),b(\xi)-b(\eta)\>^+ + \|\si(\xi)-\si(\eta)\|_{\rm HS}^2 \le L_0\|\xi-\eta\|_r^2,\ \ \xi,\eta\in \C_r. $$
\item[({\bf A2})] For each $\xi\in\C_r$, $\si(\xi)$ is invertible,
and
 $\sup_{\xi\in\C_r}\{\|\si(\xi)\|+
 \|\si(\xi)^{-1}\|\}<\8$;
\item[({\bf A3})]
There exists a continuous function $V:\C_r\rightarrow\R_+$ with
$\lim_{\|\xi\|_r\rightarrow\8}V(\xi)=\8$ such that
\begin{equation*}
 {P}_t V(\xi)\le K\e^{-\gamma\, t}V(\xi)+K
\end{equation*}
holds for some constants $K, \gamma>0$.
\end{enumerate}

\begin{thm}\label{ex} Let
  $({\bf A0})$ and $({\bf A1})$ hold. Then   
  $\eqref{eq1} $ has a unique solution.
   Moreover, there exists a constant $C>0$ such that
  $$\E\|X_t^\xi\|_r^2\le C(1+ \|\xi\|_r^2) \e^{C\,t},\ \ t\ge 0, ~~\xi\in\C_r.$$
\end{thm}
 
 In the case   $G\equiv {\bf0}$,    the existence and
   uniqueness of solutions
  has been investigated in, e.g.,
  \cite{M08,M84,WYM} under the local Lipschitz  and
coercive conditions. When the length of memory is finite, the local
Lipschitz condition was developed in \cite{VS} by the local weak
monotone condition.
Theorem \ref{ex}  extends the result of \cite{VS} to path-dependent
SDEs  of neutral type and with infinite length of memory. Since we are interesting in 
ergodicity of the functional solutions, we are not going for giving a proof for this Theorem.

Next, we consider the ergodicity of the functional solutions to
\eqref{eq1}. In the situation of  Theorem \ref{ex},   for any
$\xi\in \C_r$, let $(X^\xi(t))_{t\ge 0}$ be the unique solution to
\eqref{eq1} with the initial datum  $X_0^\xi=\xi$. Then, the
functional solution $(X_t^\xi)_{t\ge 0}$ is a Markov process with
the semigroup
$$P_t f(\xi):= \E f(X_t^\xi)=\int_{\C_r} f(\eta)P_t(\xi,\d\eta),\ \ t\ge 0,\, f\in \B_b(\C_r), \,\xi\in \C_r,$$
where $P_t(\xi,\cdot)$ is the distribution of $X_t^\xi$.
 As explained above,  $P_t(\xi,\cdot)$ does not converges in the total variational distance. To investigate the ergodicity,  we will take the Wasserstein distance induced
 by the distance
 \begin{equation}\label{eq6}
 \rr_r(\xi,\eta):= 1\land \|\xi-\eta\|_r,\ \ \xi,\eta\in
\C_r.\end{equation} For any $\mu,\nu\in \scr P(\C_r)$, the
collection of all probability measures on $\C_r$, the
$L^1$-Warsserstein distance between $\mu$ and $\nu$ induced by
$\rr_r$ is defined by \beq\label{WS}\mathbb
W_{\rr_r}(\mu,\nu)=\inf_{\pi\in \C(\mu,\nu)} \int_{\C_r\times\C_r}
\rr_r(\xi,\eta) \pi(\d\xi,\d\eta),\end{equation} where $\C(\mu,\nu)$
is the set of all couplings of $\mu$ and $\nu$; that is, $\pi\in
\C(\mu,\nu)$ if and only if  $\pi$ is a probability measure on
$\C_r\times\C_r$ such that $\pi(\cdot\times\C_r)=\mu$ and
$\pi(\C_r\times \cdot)=\nu$.

Condition {\bf(A1)} with $G(\xi_0)={\bf 0}$ implies {\bf (H2)} and {\bf (H3)}  and   according to    Theorem \ref{ex}  yields
 the existence and uniqueness of
  \eqref{eq1}.  For the Lyapunov function $V$ in ({\bf
A3}), let
$$\rr_{r,V}(\xi,\eta):= \sqrt{\rr_r(\xi,\eta)(1+V(\xi)+V(\eta))},\ \ \xi,\eta\in \C_r.$$

\begin{thm}\label{er}
  Assume $({\bf H1})$ and $({\bf A1})$-$({\bf A3})$.
  Then $P_t$  has a unique invariant probability measure $\pi$, and there exist constants $c,\ll>0$ such that
\begin{equation}\label{QPP2}
\mathbb {W}_{\rr_{r,V}}(\mu P_t, \nu P_t)\le c\,\e^{-\ll t}\mathbb
{W}_{\rr_{r,V}}(\mu, \nu),\ \ \mu,\nu\in \scr P(\C_r),\, t\ge 0.
\end{equation}  Consequently,  there exists a constant $C>0$ such that
\beq\label{QPP3} \mathbb {W}_{\rr_{r,V}}( P_t (\xi,\cdot), \pi)\le
C\,\e^{-\ll t}\sqrt{1+V(\xi)},\ \ t\ge 0.\end{equation}
\end{thm}

The proof of Theorem \ref{er} is based on the weak Harris' theorem
developed in \cite{HMS11}; see Theorem \ref{Harris} below
for more details. 
To meet the conditions of this theorem, one has to overcome the
difficulties caused  by the infinite length of memory.

Unlike conditions   ({\bf A1}) and ({\bf A2})  which are explicitly
imposed on the coefficients,  the Lyapunov condition ({\bf A3}) is
set by means of the semigroup $P_t$ which  is
  less explicit. In many cases, one may  verify  ({\bf A3})  by using the Lyapunov
 condition
\begin{equation}\label{c9}
\mathscr{L}V(\xi)\le-\ll V(\xi)+c,~~~\xi\in \C_r
\end{equation}  for some constants $c,\ll>0$, where $\mathscr{L}$ is the extended
generator corresponding to the semigroup $( {P}_t)_{t\ge0}.$
However, as already explained before, $\mathscr{L}$ is not yet
available for the present model. Indeed, there are a number of
examples   satisfying ({\bf A3}) but not \eqref{c9}; see, e.g.,
\cite{BYY} for such an example on   path-dependent SDEs without
dissipativity. In this spirit,  we present below
  explicit conditions for
({\bf A3}).

\begin{prp}\label{pro}\label{Ly} Let $\mu_0\in \scr P((-\8,0])$ such that
\begin{equation}\label{eq10}
 \dd_r(\mu_0):=\int_{-\infty}^0
\e^{-2r\theta}\mu_0(\d\theta)<\infty,\end{equation} and set
\begin{equation}\label{eq12}
\bb:=\Big(1+\sqrt{ \aa_1  + \aa_2 \dd_r(\mu_0)}\Big)^2.
\end{equation}
Then $({\bf A3})$ holds for   $V(\xi):=\|\xi\|_r^2$ provided that
the following two conditions hold:
\begin{enumerate}
\item[$(i)$]
For any $\xi\in\C_r$,   there exist constants $\aa_1,\aa_2>0$ with
$\aa_1+\aa_2\dd_r(\mu_0)<1$ such that
\begin{equation}\label{eq11}
|G(\xi)|^2\le\aa_1|\xi(0)|^2+\aa_2\int_{-\8}^0|\xi(\theta)|^2\mu_0(\d
\theta).
\end{equation}
\item[$(ii)$]There exist  constants $c_0,\ll_1,\ll_2>0$ with
$\gamma:=\ll_1-2r\bb-\ll_2\dd_r(\mu_0)>0$ such that
\begin{equation}\label{r1}\beg{split}
&2\<\xi(0)-G(\xi),b(\xi)\>+\|\si(\xi)\|_{\rm HS}^2\le c_0
-\ll_1|\xi(0)|^2+\ll_2\int_{-\8}^0|\xi(\theta)|^2\mu_0(\d\theta),
\end{split}
\end{equation}
\begin{equation}\label{r11}
\|\si(\xi)\|_{\rm HS}^2\le
c_0\Big(1+|\xi(0)|^2+\int_{-\8}^0|\xi(\theta)|^2\mu_0(\d\theta)\Big).
\end{equation}
\end{enumerate}

\end{prp}

To conclude this section, we present below a concrete example to
illustrate Theorem \ref{er}.

\begin{exa}\label{ex1}
Let $\mu_0(\d\theta)=\ff{1}{r_0}\e^{r_0\theta}\d\theta\in\scr
P((-\8,0])$ for some $r_0>2r$ and let
$$G(\xi)=\gg_1\int_{-\8}^0\xi(\theta)\mu_0(\d\theta),\,\,\,\,\, \si(\xi)=1+\gg_2\int_{-\8}^0(1\wedge|\xi(\theta)|)\mu_0(\d\theta),$$
\begin{equation*}
b(\xi)=-\gg_3\xi(0)-\gg_4\Big(\xi(0)-\gg_1\int_{-\8}^0\xi(\theta)\mu_0(\d\theta)\Big)^{\ff{1}{3}}+\gg_5\int_{-\8}^0\xi(\theta)\mu_0(\d\theta)
\end{equation*}
for some constants $\gg_i>0, i=1,\cdots,5.$ If
\begin{equation}\label{d4}
\gg_1^2<r_0(r_0-2r)~~~\mbox{ and }~~~
2\gg_3>2r\Big(1+\ff{\gg_1}{\sqrt{r_0(r_0-2r)}}\Big)^2+\ff{\gg_2^2}{r_0(r_0-2r)}+\ff{2(\gg_5+\gg_1\gg_3)}{\sqrt{r_0(r_0-2r)}},
\end{equation}
assertions in   Theorem  \ref{er} hold.
\end{exa}

The remainder of this paper is organized as follows.  In
Section \ref{sec3},   Theorem \ref{er} is proved by using weak
Harris' theorem and asymptotic coupling. In Section 4, we prove
Proposition \ref{Ly} and Example \ref{ex1}.

\section{Proof of Theorem \ref{er}}\label{sec3}

For simplicity, we introduce the following notation.  Let $f,g\in C(\R;\R^d)$, define
\begin{equation}\label{eq4}
 \LL^f(t)= f(t)-G(f_t),\ \ \LL^{f,g}(t)= \LL^f(t)-\LL^g(t),~~~~t\ge0.
 \end{equation} By
 ({\bf A0}) and using $\vv=\ff\aa{1-\aa}$,
we obtain
$$|f(t)-g(t)|^2\le (1+\vv) |\LL^{f,g}(t)|^2 + (1+\vv^{-1})\aa^2 \|f_t-g_t\|_r^2= \ff{|\LL^{f,g}(t)|^2}{1-\aa} +\aa\|f_t-g_t\|_r^2.$$
When $f_0=g_0$, this implies   \beg{align*} \e^{2rt}\|f_t-g_t\|_r^2
&= \sup_{0\le s\le t} (\e^{2rs} |f(s)-g(s)|^2 )\le \ff 1{1-\aa}
\sup_{0\le s\le t}(\e^{2rs}|\LL^{f,g}(s)|^2) +\aa\e^{2rt}
\|f_t-g_t\|_r^2,\end{align*} so  that  \beq\label{RPPa}
\e^{2rt}\|f_t-g_t\|_r^2\le \ff 1 {(1-\aa)^2}\sup_{0\le s\le t}
(\e^{2rs}|\LL^{f,g}(s)|^2),~~~~f_0=g_0.
\end{equation}
Similarly, \beq\label{RPPb}
\begin{split} \e^{2rt}\|f_t\|_r^2 &\le \ff{1}{1-\aa}\|f_0\|^2_r+\ff 1
{(1-\aa)^2}\sup_{0\le s\le t} (\e^{2rs}|\LL^{f}(s)|^2). 
\end{split}\end{equation}

We shall complete the proof of Theorem \ref{er} by the aid of
   weak Harris' theorem introduced in  \cite{HMS11}.
  For readers' convenience, we state it below in details. We first recall
 some notions.
\begin{defn} Let $\mathbb {X}$  be a Polish space, and  $( {P}_t)_{t\ge0}$ a Markov semigroup with transition kernel $ P_t(\xi,\cdot)$ on $\mathbb X$.  \beg{enumerate} \item[$(1)$]
A continuous function $V:\mathbb {X}\to \R_+$ is called a Lyapunov
function for  $( {P}_t)_{t\ge0}$,   if there exist constants $
\gamma, K>0$ such that
\beq\label{YPP}
 {P}_tV(\xi):=\int_\mathbb {X}V(\eta)
{P}_t(\xi,\d\eta)\le K\,\e^{-\gamma t} V(\xi)+K,~~~~~\xi\in \mathbb
{X},~~~t\ge0.
\end{equation}
\item[$(2)$]  A function $\rr:\mathbb {X}\times\mathbb
{X}\rightarrow [0,1]$ is said to be distance-like if it is
symmetric, lower semi-continuous, and $\rr(\xi,\eta)=0$ if and only
if $\xi=\eta$.
\item[$(3)$] A set
$A\subset \mathbb {X}$ is said to be $\rr$-small for $ {P}_t$, if
there exists $\varepsilon\in(0,1)$ such that
\begin{equation*}
\mathbb {W}_\rr( {P}_t(\xi,\cdot), {P}_t(\eta,\cdot))\le 1-
\vv,\quad  \xi,\eta\in A,
\end{equation*} where $\mathbb {W}_\rr$ is defined as in  \eqref{WS} for $(\mathbb X, \rr)$ replacing $(\C_r,\rr_r)$.
\item[$(4)$]   $\rr$ is said to be contractive for $ {P}_t$,    if there exists
$\vv\in(0,1)$ such that
\begin{equation*}
\mathbb {W}_\rr( {P}_t(\xi,\cdot), {P}_t(\eta,\cdot))\le \vv
\,\rr(\xi, \eta),\ \  \,\xi,\eta\in \mathbb X~\mbox{ with }
\rr(\xi,\eta)<1.
\end{equation*}\end{enumerate}
\end{defn}
The following result is   due to \cite[Theorem 4.8]{HMS11}.

\begin{thm}\label{Harris} Let $\rr$ be a distance-like function on $\mathbb X\times\mathbb X$, and  $V$  a Lyapunov function   such that $\eqref{YPP} $ holds for some constants $\gg,K>0$.   If there exists a constant $t^*>0$ such that  $\{V\le 4K\}$ is $\rr$-small and $\rr$ is contractive for $ P_{t^*}$, then   there exists a constant  $t>0$ such that
\begin{equation*}
\mathbb {W}_{\rr_V}(\mu  {P}_t, \nu  {P}_t)\le
 \ff 1 2 \mathbb {W}_{\rr_V}(\mu, \nu),~~~~\forall\,\,\mu,\nu\in\scr{P}(\mathbb {X}),
\end{equation*}
where $\rr_V(\xi,\eta):=\sqrt{\rr(\xi,\eta)(1+V(\xi)+V(\eta))},
\xi,\eta\in\mathbb X.$
\end{thm}

To apply this result to the present model,   for any $\dd>0$  and
$R>0$, let
\begin{equation*}
 \rr_{r,\dd} = 1\land
(\dd^{-1}\rr_r), ~~ ~~B_R=\{\xi\in\C_r:\ \|\xi\|_r\le R\},
\end{equation*}
\beq\label{**R} t_{R,\dd}:= 1+\ff 1 {2r}
\log\bigg(\ff{3}{2\dd^2}\Big(\ff{2\e^{2r}}{(1-\aa)^2}\Big(2R+\ff\dd
3\Big)^2+\ff{R^2}{1-\aa}\Big)\bigg),
\end{equation}
where $\rr_r$ was given in \eqref{eq6}. Obviously,  the metric
$\rr_{r,\dd}$ is equivalent to  $\rr_r$.
  To check the
conditions in Theorem \ref{Harris}   for the present setup, we need
to prepare the following four   lemmas concerned, respectively, with
the (local) irreducibility, the continuity with respect to the
initial variable, $\rho_{r,\dd}$-small property, and
$\rho_{r,\dd}$-contractive property for the Markov transition
kernel.

\begin{lem}\label{le1} Under the conditions of Theorem $\ref{er}$, for any
$R,\dd>0$,
\begin{equation}\label{a9}
\inf_{\xi\in B_R}\P(X_t^\xi\in B_\dd)>0,\ \ t\ge t_{R,\dd}.
\end{equation}
\end{lem}

\begin{proof}
The crucial  point of the proof is to apply a standard result (e.g.,
\cite[Lemma I.8.3]{Bass})  that a uniform elliptic diffusion process
is irreducible. So, below we will compare the radial process
$|X^\xi|(s)$ with an elliptic diffusion process. For any
  $\xi\in B_R$ and $\dd>0$, let $h\in C_b^\infty(\R_+;\R^d)$
such that \beq\label{HH}  h(0)= \xi(0)-G(\xi)- \ff {\dd(1-\aa)} {3}
(1,0,\cdots,0),\ \ |h|\le |h(0)|,\ \  \text{and}\  h(s)=0 \
\text{for\ } s\ge 1,\end{equation} where $\aa\in(0,1)$ was
introduced in ({\bf H1}).  According to \eqref{RPPb}, we have
\begin{equation}\label{eq5}
\begin{split}(1-\aa)^2
\sup_{0\le u\le
s}(\e^{2ru}|X^\xi(u)|^2)
\le (1-\aa)\|\xi\|^2_r+ \sup_{0\le u\le s}
(\e^{2ru}|\Lambda^{X^\xi}(u)|^2),~~~~s\ge0,
\end{split}
\end{equation}
in which $\Lambda^{X^\xi}$  was defined as in \eqref{eq4} with
$f=X^\xi$. Consider the following  radial process
\begin{equation}\label{ABC}
D(s):=|\Lambda^{X^\xi}(s)-h(s)|^2-\ff{\dd^2(1-\aa)^2}{9},~~~s\ge0.
\end{equation}
By  It\^o's formula, it follows that
\begin{equation}\label{DS}
\begin{split}
\d(\e^{2rs}D(s))&=2\,r\e^{2rs}D(s)\d s+ \e^{2rs}\d D(s)\\
&=\e^{2rs}\big\{2\,rD(s)+2\<\Lambda^{X^\xi}(s)-h(s),b(X_s^\xi)-h'(s)\>+\|\si(X_s^\xi)\|_{\rm
HS}^2\big\}\d
s\\
&\quad+2\,\e^{2rs}\<\Lambda^{X^\xi}(s)-h(s),\si(X_s^\xi)\d
W(s)\>,~~~~s\ge0.
\end{split}
\end{equation}
Define the stopping time
\begin{equation}\label{ABO}
\tau=\inf\Big\{s\ge0:\e^{2rs}|D(s)|\ge
\ff{\dd^2(1-\aa)^2}{18}\Big\}.
\end{equation}
Since $D(0)=0$ and $D(s)$ is continuous with respect to $s$, we have
$\P(\tau>0)=1.$ In terms \eqref{ABC} and \eqref{ABO}, we therefore
have
\begin{equation*}
\ff{\dd^2(1-\aa)^2}{18}\ge\e^{2rs}|D(s)|\ge|D(s)|\ge\ff{\dd^2(1-\aa)^2}{9}-|\Lambda^{X^\xi}(s)-h(s)|^2,\
\ s\in [0,\tau].
\end{equation*} As a consequence, we arrive at
\begin{equation}\label{b7}
|\Lambda^{X^\xi}(s)-h(s)|^2\ge\ff{\dd^2(1-\aa)^2}{18},~~~~~s\in[0,\tau].
\end{equation} Combining \eqref{b7}  with   ({\bf A2}),  we
obtain from \eqref{DS} that
\begin{equation}\label{t1}
\ff{\d}{\d s}\<\e^{2rs}
D(s)\>=4\,\e^{4rs}|\si^*(X_s^\xi)(\Lambda^{X^\xi}(s)-h(s))|^2\in
[c_1,c_2],\ \ s\in [0,\tau\land t]
\end{equation}
for some constants $c_2>c_1>0.$ Herein $\<\cdot\>$ means the
quadratic variation of a continuous semi-martingale and  $t\ge
t_{R,\dd}$, which is to be fixed in what follows. Next, we are going
to claim that \eqref{t1} implies that
\begin{equation}\label{v4}
\P\Big(\sup_{0\le u\le s}(\e^{2ru}|D(u)|)<
\ff{\dd^2(1-\aa)^2}{18}\Big)>0,\ \ s\ge 0.
\end{equation}
To achieve  \eqref{v4}, we extend $(\e^{2rs}D(s))_{s\in [0, \tau]}$
into $ (\e^{2rs}D(s))_{s\ge0}$  in the following manner
\begin{equation}\label{ABD}Y(s):= \e^{2r(s\land\tau)}D(s\land\tau) + {\bf1}_{\{s>\tau\}}
(W^1(s)-W^1(\tau)),\ \ s\ge 0,\end{equation} where
$(W^1(s))_{s\ge0}$ stands for the first component of
$(W(s))_{s\ge0}$. Consequently, \eqref{t1} gives that
$$\ff{\d}{\d s} \<Y(s)\>\in [c_1\land 1, c_2\lor 1].$$
By \cite[Lemma I.8.3]{Bass} and using $Y(0)=0$, this yields
\begin{equation}\label{eq7}
 \P\Big(\sup_{0\le u\le s}|Y(u)|< c \Big)>0,~~~s\,c>0.
\end{equation}Combining this with \eqref{ABO} and \eqref{ABD}, we obtain \beg{align*}
\P\Big(\sup_{0\le u\le s}(\e^{2ru}|D(u)|)<
\ff{\dd^2(1-\aa)^2}{18}\Big)&= \P\Big(\sup_{0\le u\le
s\land\tau}(\e^{2ru}|D(u)|)<
 \ff{ \dd^2(1-\aa)^2}{18}, s<\tau\Big)\\
&=\P\Big(\sup_{0\le u\le s}|Y(u)|< \ff{\dd^2(1-\aa)^2}{18}\Big) >0.
\end{align*}
So, \eqref{v4} holds true.

 By using the fundamental  inequality: $\ff 12 |u|^2 -|v|^2 \le
|u-v|^2, u,v\in\R^d,$ and recalling that $h(s)=0$ for $s\ge 1$, we
deduce from \eqref{eq5} and $\xi\in B_R$ that
\begin{equation}\label{v1}
\begin{split}
 &\P\Big(\sup_{0\le s\le t}(\e^{2rs}|D(s)|)\le \ff{\dd^2(1-\aa)^2}{18}\Big)\\
 &\le \P\Big(\sup_{0\le s\le t}(\e^{2rs} |\Lambda^{X^\xi}(s)-h(s)|^2)\le
\ff{\dd^2(1-\aa)^2}{6}\e^{2rt}\Big)\\
 &\le \P\Big(\sup_{0\le s\le t}\Big( \ff{1}{2}\e^{2rs}|\Lambda^{X^\xi}(s)|^2-\e^{2rs}|h(s)|^2\Big)\le \ff{\dd^2(1-\aa)^2}{6}\e^{2rt}\Big)\\
 &\le  \P\Big(\sup_{0\le s\le t}(\e^{2rs} |\Lambda^{X^\xi}(s)|^2)\le 2\sup_{0\le s\le 1}(\e^{2rs}
 |h(s)|^2)+\ff{\dd^2(1-\aa)^2}{3}\e^{2rt}\Big)\\
 &\le \P\Big( \sup_{0\le s\le
t}(\e^{2rs}|X^\xi(s)|^2)\le  \ff{1}{
1-\aa}R^2+\ff{2}{(1-\aa)^2}\sup_{0\le s\le 1}(\e^{2rs}
 |h(s)|^2)+\ff{\dd^2}{3}\e^{2rt}\Big).
\end{split}
\end{equation}
On the other hand, we observe that
\begin{align*}
\P(X_t^\xi\in B_\dd)&=\P\Big(\e^{-2rt}\sup_{-\8<s\le
t}(\e^{2rs}|X^\xi(s)|^2)\le \dd^2\Big)\\
&=\P\Big(\|\xi\|_r^2\vee\sup_{0\le s\le
t}(\e^{2rs}|X^\xi(s)|^2)\le\e^{2rt} \dd^2\Big)\\
& \ge\P\Big(\sup_{0\le s\le t}(\e^{2rs}|X^\xi(s)|^2)\le\e^{2rt}
\dd^2\Big).
\end{align*} Combining this with \eqref{v4} and \eqref{v1},  it
follows  that
\begin{equation}\label{w1}
 \ff{1}{
1-\aa}R^2+\ff{2}{(1-\aa)^2}\sup_{0\le s\le 1}(\e^{2rs}
 |h(s)|^2)+\ff{\dd^2}{3}\e^{2rt}\le\e^{2rt}
\dd^2,\ \ t\ge t_{R,\dd}.
\end{equation}

So it remains to prove \eqref{w1}. By \eqref{HH}, ({\bf H1}),
 $G(\xi_0)={\bf0}$ and $\xi\in B_R$,  we infer that
$$|h|^2\le |h(0)|^2\le \Big(2R+\ff \dd 3\Big)^2, $$
which incurs
$$  \sup_{0\le s\le 1}(\e^{2rs}
 |h(s)|^2) \le \e^{2r}\Big(2R+\ff \dd 3\Big)^2.$$ Then \eqref{w1}
 holds definitely
 provided
$$\ff{2\dd^2} 3 \e^{2rt} \ge \ff{2\e^{2r}}{(1-\aa)^2}\Big(2R+\ff\dd 3\Big)^2+\ff{R^2}{1-\aa},$$
which indeed  is true for $t\ge t_{R,\dd}$.
\end{proof}

\begin{lem}\label{le2} Under conditions of Theorem $\ref{er}$, there exists a constant $K>0$ such that
\begin{equation}\label{s2}
\E\|X_t^\xi-X_t^\eta\|_r^2\le K\e^{K\,t}\,\,\|\xi-\eta\|_r^2,\ \
t\ge 0,\,\, \xi,\eta\in \C_r.
\end{equation}
\end{lem}
\begin{proof} Let $\LL^{\xi,\eta}(t)=\LL^{X^\xi,X^\eta}(t)$ be defined by \eqref{eq4} with $f=X^\xi$ and $g=X^\eta$. Then, as in \eqref{RPPa} we have
\beq\label{RPP4} \e^{2rt}\|X_t^\xi-X_t^\eta\|_r^2\le  \ff 1 { 1-\aa
}\|\xi-\eta\|_r^2+\ff 1 {(1-\aa)^2}\sup_{0\le s\le t}(\e^{2r s}
 |\LL^{\xi,\eta}(s) |^2).\end{equation} Thus, to obtain the desired assertion \eqref{s2},    it is sufficient to
 show that
\beq\label{s2'}\GG(t):=\E\Big(\sup_{0\le s\le
t}(\e^{2rs}|\LL^{\xi,\eta}(s)|^2)\Big)\le J\e^{Jt}
\|\xi-\eta\|_r^2,\ \ t\ge 0,\ \xi,\eta\in \C_r\end{equation} for
some constant $J>0$. Applying It\^o's formula and using ({\bf A1})
and ({\bf A2}), we obtain
\begin{equation*}
\begin{split}
\e^{2rt}|\LL^{\xi,\eta}(t)|^2 &\le 4\|\xi-\eta\|_r^2+\int_0^t
\e^{2rs}\{2r|\LL^{\xi,\eta}(s)|^2+L_0\|X_s^\xi-X_s^\eta\|_r^2\}\d
s\\
&\quad+2\,\int_0^t\e^{2rs}\<\LL^{\xi,\eta}(s),(\si(X_s^\xi)-\si(X_s^\eta))\d
W(s)\>. \end{split} \end{equation*} Combining this with
\eqref{RPP4}, ({\bf A2}) and  BDG's inequality, we find out
constants $c_1,c_2,c_3>0$ such that
 \begin{align*}
\GG(t) &\le
4\|\xi-\eta\|_r^2+c_1\int_0^t\{\|\xi-\eta\|_r^2+\GG(s)\}\d
s\\
 &\quad+\E\Big(\sup_{0\le s\le t}(\e^{2rs}|\LL^{\xi,\eta}(s)|^2)\int_0^t\e^{2rs}\|\si(X_s^\xi)-\si(X_s^\eta)\|_{\rm HS}^2\d  s \Big)^{1/2}\\
 &\le 4\|\xi-\eta\|_r^2+c_1\int_0^t\{\|\xi-\eta\|_r^2+\GG(s)\}\d
s+\ff{1}{2}\GG(t)+c_2\int_0^t\E(\e^{2rs}\|X_s^\xi-X_s^\eta\|_r^2)\d  s \\
&\le\ff{1}{2}\GG(t)+c_3(1+t)\|\xi-\eta\|_r^2+c_3\int_0^t\GG(s)\d s.
\end{align*}
Consequently,
\begin{align*}
\GG(t)\le 2c_3(1+t)\|\xi-\eta\|_r^2+2c_3\int_0^t\GG(s)\d s.
\end{align*}
By Gronwall's inequality, we obtain
\begin{align*}
\GG(t)\le 2c_3(1+t)\e^{2c_3t}\|\xi-\eta\|_r^2.
\end{align*}
Therefore, \eqref{s2'} follows from  \eqref{RPP4} immediately.
\end{proof}

\begin{lem}\label{le3} Under conditions of Theorem $\ref{er}$, for any $R,\dd>0$,
\begin{equation}\label{eq8}
\mathbb {W}_{\rr_{r,\dd}}( {P}_t(\xi,\cdot), {P}_t(\eta,\cdot))\le
1-\ff{\aa^2_t}{2} <1,\ \ t\ge t_{R,\dd/4}, \ \xi,\eta\in B_R,
\end{equation}
holds for  $t_{R,\dd/4}$ in \eqref{**R} and $ \aa_t:= \inf_{\xi\in
B_R} \P(X_t^\xi\in B_{\dd/4}).$
\end{lem}

\begin{proof}
For any $\xi,\eta\in B_R$, let $(X_t^\xi)_{t\ge0}$ be the functional
solution to  \eqref{eq1} with the initial value
$X_0^\xi=\xi\in\C_r$, and   $(\tt X_t^\eta)_{t\ge0}$  the functional
solution to \eqref{eq1} with the initial datum  $\tt X_0^\eta=\eta$
but for an independent Brownian motion $(\tt W(t))_{t\ge0}$
replacing $(W(t))_{t\ge0}$. We call $(X_t^\xi, \tt X_t^\eta)$ an
independent coupling of the functional solutions to \eqref{eq1}. In
view of the independence of $(X_t^\xi)_{t\ge0}$ and $(\tt
X_t^\eta)_{t\ge0}$, we deduce that
\begin{equation*}
\begin{split}
\mathbb {W}_{\rr_{r,\dd}}( {P}_t(\xi,\cdot), {P}_t(\eta,\cdot))
& \le \E(1\wedge(\dd^{-1}\|X_t^\xi-X_t^\eta\|_r)) \\
& \le \ff{1}{2}\P(X_t^\xi\in B_{\dd/4},X_t^\eta\in B_{\dd/4})
 +\P(\{X_t^\xi\notin B_{\dd/4}\} \cup \{X_t^\eta\notin
B_{\dd/4}\})\\
&=1-\ff{1}{2}\P(X_t^\xi\in B_{\dd/4})\P(X_t^\eta\in B_{\dd/4}) \\
&\le 1-\ff{\aa^2_t}{2}.
\end{split}
\end{equation*}
Hence,  \eqref{eq8} holds true due to  Lemma \ref{le1}.
\end{proof}

 \begin{lem}\label{le4} Under the conditions of Theorem $\ref{er}$,   for any $\bb\in (0,1)$ there exist  constants $\dd_\bb, t_\bb>0$ such that
\begin{equation}\label{ap7}
\mathbb {W}_{\rr_{r,\dd}}( {P}_t(\xi,\cdot), {P}_t(\eta,\cdot))\le
\bb\,\rr_{r,\dd}(\xi,\eta),~~~~~t\ge t_\bb, \ \dd\in (0,\dd_\bb]
\end{equation}
for any $\xi,\eta\in\C_r$ with $\rr_{r,\dd}(\xi,\eta)<1$.

\end{lem}

\begin{proof}
Our proof is based on the Girsanov transform and  Warsserstein
coupling, which is more straightforward than the ``binding
construction'' argument adopted in \cite[p.254-257]{HMS11}.  For
$\xi,\eta\in \C_r$, let $(X_s^\xi)_{s\ge0}$ be the functional
solution to \eqref{eq1}, and let $(Y^\eta(s))_{s\ge0}$ solve  the
following SDE
\begin{equation}\label{ap0}
\d
\{Y^\eta(s)-G(Y^\eta_s)\}=\big\{b(Y_s^\eta)+\ll\LL^{\xi,\eta}(s)\big\}\d
s+\si(Y_s^\eta)\d W(s),~~~~s\ge0,~~~~Y_0^\eta=\eta,
\end{equation} where $\ll>0$ is a constant,  $\LL^{\xi,\eta}(s):=\LL^{X^\xi,X^\eta}(s)$ is defined   in
\eqref{eq4}.
For $\ll>0$ sufficiently large, the additional drift
$\ll\,\LL^{\xi,\eta}(s)$ strongly   pushes  $Y_s^\eta$ moving toward
to $X_s^\xi$ whenever $s\uparrow\8$. Indeed,    when $\ll>0$ is
sufficiently large, for any $r_0\in(0,r)$, there exists a constant
$c>0$ such that
\begin{equation}\label{ap2}
\E\|X^\xi_s-Y^\eta_s\|_r^2\le c\,\e^{-r_0s}\|\xi-\eta\|_r^2,\ \ s\ge
0,\, \xi,\eta\in \C_r,
\end{equation} and, for any stopping time $\tau$,
\beq\label{SPP} \E\|X^\xi_{s\land\tau}-Y^\eta_{s\land\tau}\|_r^2\le
c\,\|\xi-\eta\|_r^2,\ \ s\ge 0, \,\xi,\eta\in \C_r\end{equation}
  as shown in the proof of
  \cite[(3.11)]{BWY},

 To compare $Y_s^\eta$ with $X_s^\eta$ via the Girsanov theorem, let
  $h(s)=\ll\,\si^{-1}(Y_s^\eta)\LL^{\xi,\eta}(s)$  and set
  \begin{equation*}
R_t:=\exp\Big(-\int_0^t\<h(s),\d
W(s)\>-\ff{1}{2}\int_0^t|h(s)|^2\d s\Big).
\end{equation*} Generally, $R_t$ may not  be  a well defined probability density, so  we shall restrict it by the following stopping
time
\begin{equation*}
\tau_\vv=\inf\bigg\{s\ge0: \int_0^s|h(s)|^2\d
s\ge\vv^{-1}\|\xi-\eta\|_r^2\bigg\}
\end{equation*}
for some constant $\vv\in(0,1)$ sufficiently small to be determined
later. By Girsanov theorem, $\d\Q_{\vv}:=  R_{t\land\tau_\vv}\d\P$
is a probability measure on $(\OO,\F)$  under which
\begin{equation*}
\tt W(s) := W(s) + \int_0^{s\land \tau_\vv}  h(u) \d u,\ \ s\ge 0
\end{equation*} is a $d$-dimensional Brownian motion. Let
$\tt Y^\eta(s)$ solve  the SDE
\begin{equation*}
\d \{\tt Y^\eta(s)-G(\tt Y^\eta_s)\}=\big\{b(\tt Y_s^\eta)+{\bf 1}_{\{\tau_\vv\ge s\}}\ll\,\tt\LL^{\xi,\eta}(s)\big\}\d s+\si(\tt Y_s^\eta)\d
\,W(s),~~~~s\ge 0,~~~~\tt Y_0^\eta=\eta,
\end{equation*} where $\tt \LL^{\xi,\eta}(s):= \Lambda^{X^\xi,\tt Y^\eta}(s).
$ By   the weak uniqueness of solutions to \eqref{ap0} up to time
$t\land\tau_\vv$, we have \beq\label{XTT1} \P(X_{t}^\eta\in\cdot)=
\Q_{\vv}(\tt Y_{t}^\eta\in\cdot),\ \ \tt Y_{t\land\tau_\vv}^\eta=
Y^\eta_{t\land\tau_\vv},\ \ t\ge 0.\end{equation}

To estimate $\mathbb {W}_{\rr_{r,\dd}}( {P}_t(\xi,\cdot),
{P}_t(\eta,\cdot))$, we take the following Wasserstein coupling of
$\P$ and $\Q_\vv$:
\beq\label{SPP2} \beg{split} \Pi(\d\oo,\d\tt\oo)= & (1\land R_{t\land\tau_\vv})(\oo) \P(\d\oo)\dd_{\oo}(\d\tt\oo)\\
&+ \ff{(1-R_{t\land\tau_\vv})^+(\oo) (R_{t\land\tau_\vv}-1)^+(\tt\oo)}{\E[(1-R_{t\land \tau_\vv})^+]}\, \P(\d\oo) \P(\d\tt\oo),
\end{split}\end{equation} where $\dd_\oo$ is the Dirac measure at point $\oo$, and the last term vanishes if $\E[(1-R_{t\land \tau_\vv})^+]=0$ which is only possible when $\xi=\eta$.
Combining  this coupling  with \eqref{XTT1}, and noting that
$\rr_{r,\dd}\le 1$, we obtain that \beq\label{SPP3}
\beg{split}&\mathbb {W}_{\rr_{r,\dd}}( {P}_t(\xi,\cdot),
{P}_t(\eta,\cdot))\\
&\le \int_{\OO\times\OO} \rr_{r,\dd}(X_t^\xi(\oo), \tt Y_t^\eta(\tt\oo)) \Pi(\d\oo,\d\tt\oo)\\
&\le \E\big[\rr_{r,\dd}(X_t^\xi, \tt Y_t^\eta) (1\land R_{t\land\tau_\vv})\big]  +\E \big[(R_{t\land \tau_\vv}-1)^+\big]\\
&\le \E\big[{\bf 1}_{\{t\le\tau_\vv\}}
\rr_{r,\dd}(X_t^\xi,Y_t^\eta)\big] + \E\big[{\bf 1}_{\{t>\tau_\vv\}}
\rr_{r,\dd}(X_t^\xi,\tt Y_t^\eta)\big] + \E \big[(R_{t\land
\tau_\vv}-1)^+\big]\\
&=:I_1(t)+I_2(t)+I_3(t).\end{split}\end{equation} Next we are going
to estimate   three terms above, one-by-one.

Firstly, by \eqref{ap2},  there exists a constant $c>0$ such that
\beq\label{YTT1} I_1(t)\le c\,  \e^{-r_0t/2} \dd^{-1}\|\xi-\eta\|_r=
c\, \e^{-r_0t/2} \rr_{r,\dd}(\xi,\eta) \end{equation} for arbitrary
$\xi,\eta\in\C_r$ with $\rr_{r,\dd}(\xi,\eta)<1.$ Next,    by
H\"older's inequality, the strong Markov property, \eqref{s2},
 \eqref{ap2}, \eqref{SPP},  $\tt Y^\eta_{t\land
\tau_\vv}= Y^\eta_{t\land\tau_\vv}$ due to \eqref{XTT1}, and noting
that the SDE for $\tt Y_{s}^\eta$ coincides with \eqref{eq1} when
$s\ge \tau_\vv$, we obtain that
\begin{equation}\label{ap4}
\begin{split}I_2(t) &\le\dd^{-1}\E\big[\|X_t^\xi-\tt Y_t^\eta\|_r{\bf
1}_{\{\tau_\vv<t\}}\big]\\
&=\dd^{-1} \E\Big[{\bf
1}_{\{\tau_\vv<t\}}\big\{\E\|X_{t-\tau_\vv}^{\xi'}-\tt
Y_{t-\tau_\vv}^{\eta'}\|_r\big\}
\big|_{(\xi',\eta')=(X_{t\land\tau_\vv}^\xi, Y_{t\land\tau_\vv}^\eta)}\Big]\\
&\le \dd^{-1} \sqrt{\P(\tau_\vv<t) K \e^{Kt} \E \|X_{t\land\tau_\vv}^\xi-Y_{t\land\tau_\vv}^\eta\|_r^2}\\
& \le \sqrt{cK\e^{Kt}\P(\tau_\vv<t)}\,\rr_{r,\dd}(\xi,\eta)
\end{split}
\end{equation} for any $\xi,\eta\in\C_r$ with
$\rr_{r,\dd}(\xi,\eta)<1$. On the other hand,    Chebyshev's
inequality, ({\bf H1}), ({\bf A1}), ({\bf A2}) and
 \eqref{ap2} imply
\begin{equation}\label{ap5}
\begin{split}
\P(\tau_\vv<t) &\le \P\bigg(\int_0^t|h(s)|^2\d s
 \ge\vv^{-1}\|\xi-\eta\|_r^2\bigg)\\
&\le c_1\vv\,\|\xi-\eta\|_r^{-2} \int_0^t\E\|X^\xi_s-Y^\eta_s\|_r^2\d s\\
&
 \le c_2\,\vv  \int_0^t\e^{-r_0s}\d s \le\ff{c_2\,\vv}{r_0}
\end{split}
\end{equation}
for some constants $c_1,c_2>0.$ Combining \eqref{ap5} with
\eqref{ap4}, we may find out a constant $c_3>0$ such that
\beq\label{SPP4} I_2(t)\le c_3\sqrt\vv \e^{c_3t}
\rr_{r,\dd}(\xi,\eta),~~~\xi,\eta\in\C_r,~~\rr_{r,\dd}(\xi,\eta)<1.
\end{equation}  By H\"older's inequality and the
definition of the stopping time $\tau_\vv$, we obtain
\begin{equation*}
\begin{split}
I_3^2(t)&\le\E R_{t\land \tau_\vv}^2-1\\
&\le\E\exp\Big(-2\int_0^{t\land \tau_\vv}\<h(s),\d W(s)\>-
\int_0^{t\land \tau_\vv}|h(s)|^2\d s\Big)-1\\
&\le\bigg(\E\exp\Big(6 \int_0^{t\land \tau_\vv}|h(s)|^2\d
s\Big)\bigg)^{1/2}-1\\
&\le
\e^{3\vv^{-1}\|\xi-\eta\|_r^2}-1\le3\vv^{-1}\|\xi-\eta\|_r^2\e^{3\vv^{-1}\|\xi-\eta\|_r^2},
\end{split}
\end{equation*}
where   the last step is due to  the inequality: $\e^x-1\le
x\,\e^x$, $x\ge0.$ Hence, for any $\xi,\eta\in\C_r$ with
$\rr_{r,\dd}(\xi,\eta)<1$, we infer that
\begin{equation*}
\begin{split}
I_3(t)\le\sqrt3\vv^{-1/2}\e^{\ff{3}{2}\vv^{-1}\dd^2}\|\xi-\eta\|_r=\sqrt3\vv^{-1/2}\dd\e^{\ff{3}{2}\vv^{-1}\dd^2}\rr_{r,\dd}(\xi,\eta).
\end{split}
\end{equation*}
 Combining this with \eqref{SPP3}, \eqref{YTT1} and \eqref{SPP4}, we arrive at
 $$\mathbb {W}_{\rr_{r,\dd}}( {P}_t(\xi,\cdot),  {P}_t(\eta,\cdot))\le
 c_3\Big\{
 \e^{-r_0t/2}
 + \e^{Kt/2} \vv +\vv^{-1/2}\dd\e^{\ff{3}{2}\vv^{-1}\dd^2}\Big\}\rr_{r,\dd}(\xi,\eta)$$
for any $\xi,\eta\in\C_r$ with $\rr_{r,\dd}(\xi,\eta)<1$. Thus,
\eqref{ap7} holds by taking $t>0$ sufficiently large and
$\vv=\dd\in(0,1)$ sufficiently small.
  \end{proof}

\begin{proof}[Proof of Theorem $\ref{er}$] Since $\lim_{\|\xi\|_r\to\infty} V(\xi)=\infty$, there is a constant $R>0$ such that
$\{V\le 4K\}\subset B_R$.   By Lemmas  \ref{le2} and \ref{le3},
there exists  $t_0\ge t_{R,\dd/4}$ such that $\{V\le 4K\}$ is
$\rr_{r,\dd}$-small and $\rr_{r,\dd}$ is contractive for $ P_{t}$
for any  $t\ge t_0$ and $\dd>0$. So, in terms of Theorem
\ref{Harris}, $ P_t$ has a unique probability measure $\pi $, and
there exists a constant $t_1>0$ such that
 \beq\label{QPP}     \mathbb {W}_{\rr_{r,\dd,V}}(\mu  {P}_{t_1},\nu  {P}_{t_1})\le
\ff 1 2 \mathbb {W}_{\rr_{r,\dd,V}}(\mu,\nu),~~~~\mu,\nu\in\scr
{P}(\C_r).\end{equation} Combining this with  the semigroup
property, to prove \eqref{QPP2} it suffices to find out a constant
$C>0$ such that \beq\label{SSPP} \mathbb {W}_{\rr_{r,\dd,V}}(\dd_\xi
{P}_{t}, \dd_\eta  {P}_{t})\le C  \rr_{r,\dd,V}(\xi,\eta),\ \ t\in
[0,t_1], \xi,\eta\in \C_r.\end{equation} By   \eqref{YPP} and
\eqref{s2}, there exists from H\"older's inequality a constant $C>0$
such that
\beg{align*} \mathbb {W}_{\rr_{r,\dd,V}}(\dd_\xi  {P}_{t}, \dd_\eta  {P}_{t})&\le\E\sqrt{\rr_{r,\dd}(X_t^\xi, X_t^\eta) (1+V(X_t^\xi)+V(X_t^\eta))}\\
&\le \sqrt{ \E\rr_{r,\dd} (X_t^\xi, X_t^\eta)  \E (1+V(X_t^\xi)+V(X^\eta_t))}\\
&\le C \sqrt{\rr_r(\xi,\eta) (1+ V(\xi)+V(\eta))}\\
&=C  \rr_{r,\dd,V}(\xi,\eta)  \end{align*}
 for any $t\in [0,t_1], \xi,\eta\in
\C_r.$ Therefore, \eqref{SSPP} holds true so that  \eqref{QPP2} is
available by in addition taking the equivalence of $\rho_r$ and
$\rho_{r,\dd}$.

Next, by \eqref{YPP}, we have $\pi(V):=\int_{\C_r}V\d\pi<\infty$, so
that \eqref{QPP2}   implies
$$\mathbb {W}_{\rr_{r,V}}( {P}_t(\xi,\cdot), \pi)=  \mathbb
{W}_{\rr_{r,V}}(\dd_\xi {P}_t, \pi {P}_t)\le c\,\e^{-\ll t}
\int_{\C_r} \rr_{r,V} (\xi,\eta)\pi(\d\eta) \le C\, \e^{-\ll t}
\sqrt{1+V(\xi)}$$ for some constant $C>0$.
\end{proof}

\section{Proofs of Proposition \ref{Ly} and Example \ref{ex1}}
%
%

\begin{proof}[  Proof of Proposition \ref{Ly} ]
For  simplicity, we write  $X(t)=X^\xi(t)$ and $X_t=X^\xi_t$. By
  \eqref{RPPb}, it is sufficient to find out a constant
$c>0$ such that
\begin{equation}\label{W}
\E\Big(\sup_{0\le s\le t}(\e^{2r s}|\LL^X(s)|^2)\Big)\le c\,\Big((1+
t)\|\xi\|_r^2+\e^{2rt}\Big).
\end{equation}
By Fubini's theorem and integration by substitution, we deduce from
\eqref{eq10} that
\begin{equation}\label{r3}
\begin{split}
&\int_0^t\int_{-\8}^0 \e^{2rs}|X(s+\theta)|^2\mu_0(\d\theta)\d
s\\
&=\int_0^t\int_{-\8}^{-s}
\e^{-2r\theta}\e^{2r(s+\theta)}|X(s+\theta)|^2\mu_0(\d\theta)\d
s+\int_{-t}^0\e^{-2\theta}\int_0^{t+\theta} \e^{2rs}|X(s)|^2\d s\mu_0(\d\theta)\\
&\le\dd_r(\mu_0)\|\xi\|_r^2t+\dd_r(\mu_0)\int_0^t \e^{2rs}|X(s)|^2\d
s,
\end{split}
\end{equation}
which, together with \eqref{eq11}, leads to: for any $\vv>0$, there
exists a constant $c_\vv>0$ such that
\begin{equation*}
\begin{split}
\int_0^t\e^{2rs} |\Lambda^X(s)|^2\d s&\le\Big( 1+\vv
+(1+1/\vv)\aa_1\Big)\int_0^t\e^{2rs} |X(s)|^2\d
s\\
&\quad+(1+1/\vv)\aa_2\int_0^t\int_{-\8}^0
\e^{2rs}|X(s+\theta)|^2\mu_0(\d\theta)\d s\\
&\le c_\vv\|\xi\|_r^2t+\Big(1+\aa_1+\aa_2 \dd_r(\mu_0)+\vv +
(\aa_1+\aa_2 \dd_r(\mu_0))/\vv \Big)\int_0^t \e^{2rs}|X(s)|^2\d s.
\end{split}
\end{equation*}
Taking $\vv=(\aa_1+\aa_2 \dd_r(\mu_0))^{1/2}$, we find out a
constant $c_1>0 $ such that
\begin{equation}\label{r5}
\int_0^t\e^{2rs} |\Lambda^X(s)|^2\d s \le
c_1\|\xi\|_r^2t+\bb_1\int_0^t \e^{2rs}|X(s)|^2\d s
\end{equation}
where $\bb_1>0$ is in \eqref{eq12}. Now, by It\^o's formula, it
follows from \eqref{r1} that
\begin{equation}\label{r2}
\begin{split}
\e^{2rt}| \Lambda^X(t)|^2 &\le|\Lambda^\xi(0)|^2+\int_0^t
\e^{2rs}\Big\{c_0+2r| \Lambda^X(s)|^2
-\ll_1|X(s)|^2\\
&\quad+\ll_2\int_{-\8}^0|X(s+\theta)|^2\mu_0(\d\theta)\Big\}\d
s+2\int_0^t\e^{2rs}\< \Lambda^X(s),\si(X_s)\d W(s)\>.
\end{split}
\end{equation}
Plugging \eqref{r3} and \eqref{r5} into \eqref{r2} and utilizing
({\bf H1}) gives
\begin{equation*}
\begin{split}
\e^{2rt}\E| \Lambda^X(t)|^2 &\le
c_2\Big((1+t)\|\xi\|_r^2+\e^{2rt}\Big)-\Big(\ll_1-2r\beta_1-\ll_2\dd_r(\mu_0)\Big)\int_0^t
\e^{2rs}\E|X(s)|^2\d s
\end{split}
\end{equation*}
for some constant $c_2>0$. Since
$\ll_1-2r\beta_1-\ll_2\dd_r(\mu_0)>0$,  this implies
\begin{equation}\label{d2}
 \int_0^t \e^{2rs}\E|X(s)|^2\d s\le c_3(1+t)\|\xi\|_r^2+c_3\e^{2rt}
\end{equation}
for some constant $c_3>0$. On the other hand, by BDG's inequality,
we deduce from \eqref{r11} and \eqref{r3} that
\begin{equation}\label{d1}
\begin{split}
&2\sup_{0\le s\le t}\Big|\int_0^s\e^{2ru}\< \Lambda^X(u),\si(X_u)\d
W(u)\>\Big|\\
&\le8\sqrt2\,\E\Big(\sup_{0\le s\le
t}(\e^{2rs}|\Lambda^X(s)|^2)\int_0^t\e^{2rs} \|\si(X_s)\|_{\rm
HS}^2\d s\Big)^{1/2}\\
&\le\ff{1}{2}\E\Big(\sup_{0\le s\le t}(\e^{2r
s}|\LL^X(s)|^2)\Big)+c_4\Big(\e^{2rt}+t\|\xi\|_r^2+\int_0^t
\e^{2rs}\E|X(s)|^2\d s\Big)
\end{split}
\end{equation}
for some $c_4>0$. Thus, by taking \eqref{r5} and \eqref{r2} into
account and making use of \eqref{d1} and ({\bf H1}), there exists a
constant $c_5>0$ such that
\begin{equation*}
\begin{split}
\E\Big(\sup_{0\le s\le t}(\e^{2r s}|\LL^X(s)|^2)\Big)\le
c_5\Big((1+t) \|\xi\|_r^2+\e^{2rt}+\int_0^t \e^{2rs}\E|X(s)|^2\d
s\Big).
\end{split}
\end{equation*}
Henceforth, \eqref{W} follows directly from \eqref{d2}.
\end{proof}

\begin{proof}[Proof of Example \ref{ex1}] It suffices  to verify ({\bf
H1}), ({\bf A1})-({\bf A2}) and conditions in Proposition \ref{pro}.

  By H\"older's inequality, one
has
\begin{equation}\label{ex2}
|G(\xi)-G(\eta)|^2\le\gg_1^2\int_{-\8}^0|\xi(\theta)-\eta(\theta)|^2\mu_0(\d\theta)\le
\gg_1^2\dd_r(\mu_0)\|\xi-\eta\|_r^2,~~~\xi,\eta\in\C_r.
\end{equation}
Therefore, ({\bf H1}) holds for $\aa=\gg_1\sqrt{\dd_r(\mu_0)}<1$ owing
to \eqref{d4}. By H\"older's inequality and \eqref{ex2}, we can find
some constants $c_1,c_2>0$ such that
\begin{equation*}
\begin{split}
&\<\xi(0)-\eta(0)-(G(\xi)-G(\eta)),b(\xi)-b(\eta)\>^++|\si(\xi)-\si(\eta)|^2\\
&\le\Big((\xi(0)-\eta(0)-(G(\xi)-G(\eta)))\Big\{-\gg_3(\xi(0)-\eta(0))+\gg_5\int_{-\8}^0(\xi(\theta)-\eta(\theta))\mu_0(\d\theta)\Big\}\Big)^+\\
&\quad+\gg_2^2\Big(\int_{-\8}^0|\xi(\theta)-\eta(\theta)| \mu_0(\d\theta)\Big)^2\\
&\le c_1\Big\{|\xi(0)-\eta(0)|^2+|G(\xi)-G(\eta)|^2 \Big\}\\
&\le c_2\|\xi-\eta\|_r^2,~~~~\xi,\eta\in\C_r,
\end{split}
\end{equation*}
where the first inequality is due to
\begin{equation*}
-\gg_4\<\xi(0)-G(\xi)-(\eta(0)-G(\eta)),(\xi(0)-G(\xi))^{1/3}-(\eta(0)-G(\eta))^{1/3}\>\le0.
\end{equation*}
Consequently, ({\bf A1}) holds true. According to the formula of
$\si(\xi)$, ({\bf A1}) holds trivially.

Finally, we verify conditions in Proposition \ref{pro}. Obviously,
\eqref{eq10} holds with $\dd_r(\mu_0)=\ff{1}{r_0(r_0-2r)}<\8$ due to
$r_0>2r.$ Thanks to $G(\xi_0)={\bf0}$ and \eqref{ex2}, \eqref{eq11}
holds for $\aa_1=0$ and $\aa_2=\gg_1^2$. By H\"older's inequality,
it follows that
\begin{equation}\label{s1}
|\si(\xi)|^2\le1+\ff{1}{\aa}
+(1+\aa)\gg_2^2\int_{-\8}^0|\xi(\theta)|^2\mu_0(\d\theta),~~~\aa>0,\,
\xi\in\C_r.
\end{equation}
Next, by using \eqref{s1},   for any $\vv>0 $ and $\xi\in\C_r$,
\begin{equation*}
\begin{split}
&2(\xi(0)-G(\xi))b(\xi)+|\si(\xi)|^2\\
&\le1+\ff{1}{\aa}-2\gg_3\xi^2(0)+2(\gg_5+\gg_1\gg_3)\xi(0)\int_{-\8}^0\xi(\theta)\mu_0(\d\theta)+
(1+\aa)\gg_2^2\int_{-\8}^0|\xi(\theta)|^2\mu_0(\d\theta)\\
&\le1+\ff{1}{\aa}-\Big(2\gg_3-(\gg_5+\gg_1\gg_3)\vv\Big)\xi^2(0)
+\Big(\ff{1}{\vv}(\gg_5+\gg_1\gg_3)
+(1+\aa)\gg_2^2\Big)\int_{-\8}^0|\xi(\theta)|^2\mu_0(\d\theta).
\end{split}
\end{equation*}
Taking $\vv=\sqrt{\dd_r(\mu_0)}$ and  $\aa\in(0,1)$ sufficiently small
and combining \eqref{d4}, we conclude that \eqref{r1} holds.

\end{proof}


\begin{thebibliography}{17}
{\small

\setlength{\baselineskip}{0.14in}
\parskip=0pt

\bibitem{BGM}
 Banerjee, S., Gordina, M.,   Mariano, P., Coupling in the Heisenberg group and its applications to gradient
 estimates,  arXiv:1610.06430v3.
\bibitem{BWY11}      Bao, J.,  Wang, F.-Y., Yuan,  C.,  Derivative formula and Harnack inequality for degenerate functional SDEs,    {\it
Stoch. Dyn.},
  {\bf 13} (2013),   1250013, 22pp.

 \bibitem{BWY11b}   Bao, J.,  Wang, F.-Y., Yuan,  C., \emph{Bismut Formulae and Applications for Functional SPDEs,} {\it  Bull. Math. Sci.},
   {\bf 137} (2013), 509--522.


\bibitem{BWY}Bao, J., Wang, F.-Y., Yuan, C., Asymptotic Log-Harnack Inequality and Applications for Stochastic
Systems of Infinite Memory, arXiv: 1710.01042v1.

\bibitem{BYY} Bao, J., Yin, G., Yuan, C., Stationary
distributions for retarded stochastic differential equations without
dissipativity,
 {\it Stochastics}, {\bf 89} (2017),   530--549.


\bibitem{Bass} Bass, R., Diffusions and elliptic operators,
Springer, New York, 1998.


\bibitem{B14}Butkovsky, O., Subgeometric rates of convergence of Markov
processes in the Wasserstein metric,  {\it Ann. Appl. Probab.},
{\bf24} (2014), 526--552.


\bibitem{BS}  Butkovsky, O., Scheutzow, M., Invariant measures for stochastic functional
differential equations, {\it Electron. J. Probab.}, {\bf22} (2017),
  1--23

\bibitem{CL} Chen, M.-F., Li, S.,   Coupling methods for multidimensional
diffusion processes,  {\it Ann. Probab.}, {\bf17} (1989),  151--177.

\bibitem{CH} Cloez, B. and Hairer, M.: Exponential ergodicity for Markov processes with random switching, {\it Bernoulli}, {\bf21} (2015),   505--536.

\bibitem{DZ} Da Prato, G.,  Zabczyk, J.,  Ergodicity for Infinite-Dimensional
Systems, in: London Mathematical Society, Lecture Note Series, vol.
229, Cambridge University Press, Cambridge, 1996.

\bibitem{DMS13} Dolbeault, J.,   Mouhot, C.,  Schmeiser, C.,
     Hypocoercivity for linear kinetic equations conserving mass,
         {\it Trans. Amer. Math. Soc.}, {\bf  367} (2015),  3807--3828.


\bibitem{GS} Grothaus, M.,   Stilgenbauer, P.,  Hypocoercivity for kolmogorov backward evolution equations and applications,  {\it J. Funct. Anal.},
 {\bf 267} (2014), 3515--3556.
 
\bibitem{GW} Grothaus, M.,  Wang,  F. -Y., Weak Poincar\'e inequalities for convergence rate of degenerate diffusion processes,   arXiv:1703.04821.
\bibitem{H16}  Hairer, M.,Convergence of Markov Processes,
http://www.hairer.org/notes/Convergence.pdf.

\bibitem{H08}  Hairer, M.,  Ergodicity for stochastic PDEs, 2008.
www.hairer.org/notes/Imperial.pdf


\bibitem{HM} Hairer, M., Mattingly,
 Jonathan C., Ergodicity of the 2D Navier-Stokes equations with degenerate stochastic forcing, {\it Ann. of Math.},   {\bf164} (2006),   993--1032.

\bibitem{HMS11}Hairer, M., Mattingly,  J.~C.,    Scheutzow, M.,   Asymptotic
coupling and a general form of Harris theorem with applications to
stochastic delay equations,  {\it Probab. Theory Related Fields},
{\bf 149} (2011),  223--259.



\bibitem{K99} Krylov, N.~V., On Kolmogorov's equations for finite-dimensional
diffusions, Stochastic PDE's and Kolmogorov equations in infinite
dimensions (Cetraro, 1998), Lecture Notes in Math., vol. 1715,
Springer, Berlin, 1999, pp. 1--63.



\bibitem{MT} Majda, A.~J., Tong, X.~T., Geometric ergodicity for piecewise contracting processes with
 applications for tropical stochastic lattice models, {\it Commu. Pure Appl. Math.}, {\bf LXIX} (2016), 1110--1153.





\bibitem{M08} Mao, X., \emph{Stochastic Differential Equations and
Applications}, Second Ed., Horwood Publishing Limited, Chichester,
2008.

\bibitem{MSH}  Mattingly, J.~C.,  Stuart,  A.~M.,     Higham, D.~J.,  Ergodicity for
SDEs and approximations: locally Lipschitz vector fields and
degenerate noise, {\it Stochastic Process. Appl.}, {\bf101} (2002),
185--232.


\bibitem{MT93} Meyn, S.~P., Tweedie, R.~L., \emph{Markov Chains and Stochastic Stability},
Springer-Verlag, Berlin, 1993.

\bibitem{M84}  Mohammed, S.-E. A., \emph{Stochastic Functional Differential
Equations}, Pitman, Boston, 1984.

 \bibitem{PR}Pr\'{e}v\^ot, C., R\"ockner, M.,
  A concise course on stochastic partial differential equations. Lecture Notes in Mathematics, 1905. Springer, Berlin, 2007.





 \bibitem{TM}Tong, X.~T., Majda, A.~ J.,
 Moment bounds and geometric ergodicity of diffusions with random switching and unbounded transition rates,
  {\it Res. Math. Sci.}, {\bf3} (2016), Paper No. 41, 33 pp.

 \bibitem{Vil09}  Villani, C.,
 Hypocoercivity, {\it Mem. Amer. Math. Soc.},
   {\bf 202 }(2009), iv+141.

\bibitem{VS}
von Renesse, M.-K., Scheutzow, M., Existence and uniqueness of
solutions of stochastic functional differential equations, {\it
Random Oper. Stoch. Equ.}, {\bf18} (2010),  267--284.

\bibitem{Wang04}  Wang, F.-Y., \emph{Functional Inequalities, Markov Semigroups and
Spectral Theory}, Science Press, 2005, Beijing.

\bibitem{Wbook} Wang, F.-Y., \emph{Harnack inequalities and
Applications for Stochastic Partial Differential Equations},
Springer, 2013, Berlin.

\bibitem{W17} Wang, F.-Y.,   Hypercontractivity and applications for stochastic Hamiltonian systems,  {\it J. Funct. Anal.}, {\bf 272} (2017),
 5360--5383.
 
\bibitem{WY10}  Wang, F.-Y.,  Yuan, C.,  Harnack inequalities for functional SDEs with multiplicative noise and applications,  {\it
 Stochastic Process. Appl.}, {\bf 121} (2011), 2692--2710.
 
\bibitem{WYM}Wu, F., Yin, G., Mei, H., Stochastic functional
differential equations with infinite delay: existence and uniqueness
of solutions, solution maps, Markov properties, and ergodicity, {\it
J. Differential Equations}, {\bf262} (2017),   1226--1252.



}
\end{thebibliography}
\end{document}